\documentclass[a4paper,dvipsnames]{article}

\usepackage[utf8]{inputenc}
\usepackage[T1]{fontenc}
\usepackage{geometry}
\usepackage[english]{babel}
                              
\usepackage{amsmath,amssymb,amsthm}
\usepackage{mathrsfs}
\usepackage{dsfont}
\usepackage[hidelinks]{hyperref}
\usepackage{graphicx}
\usepackage{ulem}
\usepackage{marvosym}
\usepackage{subfig}

\usepackage{tikz}
\usepackage{eurosym}
\usetikzlibrary{arrows}
\usetikzlibrary{decorations}
\usepackage{smartdiagram}
    \usetikzlibrary{decorations.pathmorphing}
    \usetikzlibrary{decorations.pathreplacing}
    \usetikzlibrary{decorations.shapes}
    \usetikzlibrary{decorations.text}
    \usetikzlibrary{decorations.markings}
    \usetikzlibrary{decorations.fractals}
    \usetikzlibrary{decorations.footprints}
    \usetikzlibrary{calc,trees,positioning,shadows,arrows,chains,shapes.geometric,%
    decorations.pathreplacing,decorations.pathmorphing,shapes,%
    matrix,shapes.symbols}

\definecolor{ccqqqq}{rgb}{0.8,0,0}
\definecolor{qqzzqq}{rgb}{0,0.6,0}
\definecolor{ttttff}{rgb}{0.2,0.2,1}

\newtheorem{thm}{Theorem}
\newtheorem{lem}[thm]{Lemma}
\newtheorem{prop}[thm]{Proposition}
\newtheorem{cor}[thm]{Corollary}
\theoremstyle{definition}
\newtheorem*{dfn}{Definition}
\newtheorem{asm}[thm]{Assumption}

\newtheorem*{note}{Notation}
\theoremstyle{remark}
\newtheorem{rem}{Remark}
\newtheorem*{ex}{Example}

\newcommand{\N}{\mathbb{N}}
\newcommand{\R}{\mathbb{R}}
\newcommand{\Q}{\mathbb{Q}}

\newcommand{\A}{\mathbb{A}}
\newcommand{\B}{\mathbb{B}}

\renewcommand{\S}{\mathbb{S}}

\newcommand{\bOm}{\mathbf{\Omega}}

\newcommand{\bK}{\mathbf{K}}
\newcommand{\bM}{\mathbf{M}}

\newcommand{\bB}{\mathbf{B}}

\newcommand{\bA}{\mathbf{A}}

\newcommand{\mA}{\mathrm{A}}

\newcommand{\mL}{\mathrm{L}}
\newcommand{\mM}{\mathrm{M}}

\newcommand{\va}{\mathbf{a}}
\newcommand{\vb}{\mathbf{b}}
\newcommand{\vc}{\mathbf{c}}
\newcommand{\vx}{\mathbf{x}}
\newcommand{\vy}{\mathbf{y}}
\newcommand{\ve}{\mathbf{e}}

\newcommand{\vg}{\mathbf{g}}

\newcommand{\vk}{\mathbf{k}}
\newcommand{\vl}{\mathbf{l}}

\newcommand{\vO}{\mathbf{0}}

\newcommand{\vz}{\mathbf{z}}
\newcommand{\vp}{\mathbf{p}}
\newcommand{\vq}{\mathbf{q}}

\newcommand{\vu}{\mathbf{u}}
\newcommand{\vv}{\mathbf{v}}
\newcommand{\vw}{\mathbf{w}}

\newcommand{\bmu}{\boldsymbol\mu}

\newcommand{\bvarphi}{\boldsymbol\varphi}
\newcommand{\bpsi}{\boldsymbol\psi}

\newcommand{\bsigma}{\boldsymbol\sigma}

\newcommand{\cX}{\mathcal{X}}

\newcommand{\cM}{\mathcal{M}}
\newcommand{\cC}{\mathcal{C}}

\newcommand{\cK}{\mathcal{K}}

\newcommand{\cP}{\mathcal{P}}

\newcommand{\vol}{\mathrm{vol}}

\newcommand{\op}{\mathrm{p}}

\newcommand{\od}{\mathrm{d}}

\newcommand{\x}{\times}

\newcommand{\st}{\mathrm{s.t.}}
\newcommand{\GM}{\mathrm{GM}}

\newcommand{\norm}[1]{{\left\vert\kern-0.25ex\left\vert\kern-0.25ex\left\vert #1 
    \right\vert\kern-0.25ex\right\vert\kern-0.25ex\right\vert}}

\renewcommand{\emph}[1]{{\it #1}}

\title{Convergence of Lasserre's hierarchy: the general case}
\author{Matteo Tacchi$^{1,2}$}

\begin{document}

\maketitle

\small
$^1$ École Polytechnique Fédérale, CH-1015 Lausanne, Switzerland.

$^2$ R\'eseau de Transport d'Electricit\'e (RTE), 7c place du D\^{o}me, Immeuble Window, 92073 Paris la Défense cedex, France.

\Letter \ \texttt{matteo.tacchi@epfl.ch} 

\normalsize

\begin{abstract}
Lasserre's moment-SOS hierarchy consists in approximating instances of the generalized moment problem (GMP) with moment relaxations and sums-of-squares (SOS) strenghtenings that boil down to convex semidefinite programming (SDP) problems. Due to the generality of the initial GMP, applications of this technology are countless, and one can cite among them the polynomial optimization problem (POP), the optimal control problem (OCP), the volume computation problem, stability sets approximation problems, and solving nonlinear partial differential equations (PDE). The solution to the original GMP is then approximated with finite truncatures of its moment sequence. For each application, proving convergence of these truncatures towards the optimal moment sequence gives valuable insight on the problem, including convergence of the relaxed values to the original GMP's optimal value. This note proposes a general proof of such convergence, regardless the problem one is faced with, under simple standard assumptions. As a byproduct of this proof, one also obtains strong duality properties both in the infinite dimensional GMP and its finite dimensional relaxations.

\begin{center}
\textbf{Keywords}
\end{center}
Infinite dimensional optimization; semidefinite programming; Lasserre hierarchy; numerical analysis; convex optimization; duality theory.

\begin{center}
\textbf{Declarations}
\end{center}
\textbf{Funding:} This work was funded by the French Company Réseau de Transport d'Électricité.
\end{abstract}

\section*{Acknowledgements}

The author would like to thank Didier Henrion and Jean Bernard Lasserre for fruitful discussions.

\newpage

\section{Introduction} \label{sec:intro}

$\begin{array}{llll}\text{Let:} \qquad \qquad & \bullet \ \bK \subset \R^n, & \hspace*{-2em} \bullet \ \varphi_\alpha \in \R[\vx], \alpha \in \A, & \qquad \bullet \ a_\alpha \in \R, \alpha \in \A,\\
& \bullet \ c \in \R[\vx], & \hspace*{-2em} \bullet \ \psi_\beta \in \R[\vx], \beta \in \B, & \qquad \bullet \ b_\beta \in \R, \beta \in \B, \\
& \bullet \ \A,\B \text{ be countable index sets}, \\
\end{array}$

and consider the generalized moment problem (GMP), as stated in \cite[(1.1)]{lasserre}:

\begin{alignat}{4}
\op^\star_\GM := \ && \sup \ & \int c \; d\mu \label{opt:gmpo} \\
&& \st \ & \ \ \ \mu \in \cM(\bK)_+ \nonumber \\
&&&\int \varphi_\alpha \; d\mu = a_\alpha && \qquad \alpha \in \A \nonumber \\
&&& \int \psi_\beta \; d\mu \leq b_\beta && \qquad \beta \in \B, \nonumber
\end{alignat}
{where $\cM(\bK)_+$ denotes the cone of (nonnegative) Borel measures supported on $\bK$, whose dual cone (w.r.t. the weak-$\ast$ topology on measures) is the set $\cC(\bK)_+$ of nonnegative continuous functions on $\bK$.

\begin{ex}[Generality of the problem]
\begin{itemize}
\item[]
\item The $\bK$-moment problem \cite[(3.1),(3.2)]{lasserre} is a particular instance of Problem \eqref{opt:gmpo}:
\begin{alignat*}{4}
\op^\star_\vz := \ && \sup \ & 0 \\
&& \st \ & \ \ \ \mu \in \cM(\bK)_+ \\
&&& \int \vx^\vk \; d\mu(\vx) = z_\vk && \qquad \vk \in \Gamma,
\end{alignat*}
where $\Gamma \subset \N^n$. Feasibility of this problem means existence of a measure $\mu$ whose moments on $\bK$ coincide with the $z_\vk \in \R$, $\vk \in \Gamma$.
\item The polynomial optimization problem (POP) \cite[(4.1)]{pop} is another instance of Problem \eqref{opt:gmpo}:
\begin{alignat*}{3}
\op^\star_f := \ && -\sup \ & -\int f \; d\mu \\
&& \st \ & \ \ \ \mu \in \cM(\bK)_+ \\
&&& \int 1 \; d\mu = 1,
\end{alignat*}
where $f \in \R[\vx]$. One can prove that $\op^\star_f = \inf\{f(\vx) : \vx \in \bK\}$.
\item The volume computation problem \cite{volume} is also an instance of Problem \eqref{opt:gmpo}:
\begin{alignat*}{4}
\op^\star_\bK := \ && \sup \ & \int 1 \; d\mu \\
&& \st \ & \ \ \ \mu \in \cM(\bK)_+ \\
&&& \int p \; d\mu \leq \int_\bB p(\vx) \; \od\vx && \qquad p \in \Q[\bB]_+,
\end{alignat*}
where $\bK \subset \bB \subset \R^n$ such that the $\int_\bB \vx^\vk \; \od\vx$ are known and $\Q[\bB]_+ := \{p \in \Q[\vx] : \forall \vx \in \bB,\ p(\vx) \geq 0\}$. One can prove that $\op^\star_\bK = \vol \, \bK$ is the Lebesgue volume of $\bK$.
\end{itemize}
\end{ex}
}

In 2014, the authors of \cite{joszdual} proved that under specific conditions, the Lasserre hierarchy associated with the particular GMP instance known as the polynomial optimization problem (POP) has the strong duality property, which means that there is no duality gap between the moment relaxations and the sum-of-squares strengthenings that form the moment-SOS hierarchy.

{Moreover, it is proved in \cite[Theorem 5.6(b)]{lasserre} that if the POP has a unique minimizer $\vx^\star$, then the solutions to the corresponding moment hierarchy converge towards the moments of the Dirac measure $\delta_{\vx^\star}$ such that for all Borel set $\bA \subset \R^n$, $\delta_{\vx^\star}(\bA) = \left\{\begin{array}{l}
1 \ \text{if } {\vx^\star} \in \bA \\
0 \ \text{else.}
\end{array}\right.$
}

This technical note is a follow-up that intends to close the subject of strong duality and convergence of the Lasserre hierarchy, by proving a general strong result at the abstraction level of the \emph{generalized moment problem}.

First, it is possible to slightly generalize Problem \eqref{opt:gmpo} by allowing for multiple decision variables:

\begin{alignat}{4}
\op^\star_\GM := \ && \sup \ & \int \vc \cdot d\bmu \label{opt:gmp} \\
&& \st \ & \ \ \ \bmu \in \cM(\bK_1)_+ \x \ldots \x \cM(\bK_N)_+ \nonumber \\
&&&\int \bvarphi_\alpha \cdot d\bmu = a_\alpha && \qquad \alpha \in \A \nonumber \\
&&& \int \bpsi_\beta \cdot d\bmu \leq b_\beta && \qquad \beta \in \B. \nonumber
\end{alignat}
where $\B$ is finite, $\vc, \bvarphi_\alpha, \bpsi_\beta \in \R[\vx_1]\x\ldots\x\R[\vx_N]$ and
$$ \int \vc \cdot d\bmu := \sum_{i=1}^N \int c_i \; d\mu_i. $$

\begin{rem}[Finitely many inequality constraints]
\begin{itemize}
\item[]
\end{itemize}
The introduction of multiple decision variables allows one to restrict to finitely many inequality constraints. Indeed, in the context of the moment-SOS hierarchy, infinite dimensional inequality constraints are always under the form $\nu - \Phi \bmu \in \cM(\bOm)_+$ for some compact $\bOm \subset \R^n$ and some bounded operator $\Phi :  \cM(\bK_1) \x \ldots \x \cM(\bK_N) \to \cM(\bOm)$ between vector spaces of signed measures. This means that these infinite dimensional inequality constraints can be recast as equality constraints with an additional decision variable $\bar{\mu} \in \cM(\bOm)_+$: $\Phi\bmu + \bar{\mu} = \nu$. For example, the aforementioned volume computation problem can be recast as
\begin{alignat*}{4}
\op^\star_\bK = \ && \sup \ & \int 1 \; d\mu \\
&& \st \ & \ \ \ \mu \in \cM(\bK)_+, \quad \ \ \bar{\mu} \in \cM(\bB)_+ \\
&&& \int \vx^\vk \; d\mu(\vx) + \int \vx^\vk \; d\bar{\mu}(\vx) = \int_\bB \vx^\vk \; \od\vx && \qquad \vk \in \N^n.
\end{alignat*}
\end{rem}

{\begin{ex}[A most general extension]
\begin{itemize}
\item[]
\end{itemize}
Problem \eqref{opt:gmp} covers all existing applications of the Lasserre hierarchy so far, including for instance:
\begin{itemize}
\item The optimal control problem (OCP) \cite{trelat},
\item The controlled region of attraction problem \cite{roa} (as well as its inner approximation version \cite{inneroa} and its extension \cite{outmpi} to invariant sets),
\item The reachable set problem \cite{reachable} (as well as its continuous time counterpart, formulated as the maximal positively invariant set inner approximation problem in \cite{inmpi}),
\item The nonlinear conservation PDE problem \cite[(45),(46)]{burgers}.
\end{itemize}
\end{ex}

\begin{rem}[Duality]
\begin{itemize}
\item[]
\end{itemize}
In practice, instances of Problem \eqref{opt:gmp} are formulated so that $$\forall t_1,t_2 \in \R, \alpha_1,\alpha_2,\alpha_3 \in \A, \quad t_1 \, \bvarphi_{\alpha_1} + t_2 \, \bvarphi_{\alpha_2} = \bvarphi_{\alpha_3} \Rightarrow t_1 \, a_{\alpha_1} + t_2 \, a_{\alpha_2} = a_{\alpha_3}$$ (and same for the $\bpsi_\beta$ and $b_\beta$). This compatibility with the linearity of optimization constraints allows to write a synthetic dual formulation to \eqref{opt:gmp}:
\begin{alignat}{3}
\od^\star_\GM := \ && \inf \ & \vx \cdot \va + \vy \cdot \vb \label{opt:gmpd} \\
&& \st \ & \vv + \vw - \vc \in \cC(\bK_1)_+ \x \ldots \x \cC(\bK_N)_+ \nonumber \\
&&& \vv = \sum_{\alpha \in \A} x_\alpha \, \bvarphi_\alpha \qquad \qquad \vx = (x_\alpha)_{\alpha \in \A} \in \R^\A \ \st \ \{\alpha \in \A : x_\alpha \neq 0\} \text{ is finite} \nonumber \\
&&& \vw = \sum_{\beta \in \B} y_\beta \, \bpsi_\beta \qquad \qquad \vy = (y_\beta)_{\beta \in \B} \in (\R_+)^\B \nonumber 
\end{alignat}
where $\va := (a_\alpha)_{\alpha\in\A}$ and $\vb := (b_\beta)_{\beta \in \B}$. In words, the decision variable $\vv$ is in the vector space spanned by the $\bvarphi_\alpha$ and $\vw$ is in the convex cone spanned by the $\bpsi_\beta$. $\vx$ (resp. $\vy$) is the Lagrange multiplier corresponding to the equality (resp. inequality) constraints in \eqref{opt:gmp}.
\end{rem}
}

The aim of this technical note is to state and prove a most general convergence theorem for the hierarchy corresponding to Problem \eqref{opt:gmp} {(and its dual \eqref{opt:gmpd})}. The Lasserre hierarchy technology is most often deployed under the following standard working assumption (see \cite[Assumptions 1 \& 2]{joszdual}):

\begin{asm}[Ball constraints] \label{asm:ball}
\begin{itemize}
\item[]
\end{itemize}
Suppose that there exists
$\vg_1 \in \R[\vx_1]^{m_1},\ldots,\vg_N \in \R[\vx_N]^{m_N}$ such that for $i \in {\N^\star_N} := \{1,\ldots,N\}$,
$$g_{i,1}(\vx_i) = 1 \quad ; \quad g_{i,m_i}(\vx_i) = 1 - {\|}\vx_i{\|}^2$$
$$\bK_i = \{\vx_i \in \R^{n_i} : \vg_i(\vx_i) \geq \vO\},$$
{where if $\vx = (x_1,\ldots,x_n) \in \R^n$, ${\|}\vx{\|} = \sqrt{x_1^2 + \ldots + x_n^2}$.}
\end{asm}

\begin{rem}[Compactness] \label{rem:ball}
\begin{itemize}
\item[]
\end{itemize}
Up to rescaling of the $\bK_i$'s, it is always possible to enforce Assumption \ref{asm:ball}, as soon as they are compact basic semialgebraic sets.

Indeed let $\vg \in \R[\vx]^m$, $\bK := \{\vx \in \R^n : \vg(\vx) \geq \vO\}$. For $R > 0$, define
$$\bK_R := \{\vx \in \R^n : \vg(\vx) \geq \vO, {\|}\vx{\|}^2 \leq R^2\} = \bK \cap \bB_R,$$
where $\bB_R := \{\vx \in \R^n : {\|}\vx{\|} \leq R\}$ is the ball of radius $R$. In such setting, if $\bK$ is compact then it is bounded, so that there exists $R_0 > 0$ s.t. $\forall R \geq R_0$, $\bK \subset \bB_R$, and thus $\bK = \bK_R$. This shows that if $\bK$ is compact, it is always possible to add a redundant ball constraint to its description.
\end{rem}

\begin{rem}[Archimedean property] \label{rem:archi}
\begin{itemize}
\item[]
\end{itemize}
More precisely, the standard condition for applying the moment-SOS hierarchy is actually that the $\bK_i$'s have the Archimidean property, which we explain here. Considering the cone
$$\Sigma[\vx] := \{p_1^2 + \ldots + p_k^2 : k\in\N, p_1,\ldots,p_k \in \R[\vx]\}$$
of sums of squares of polynomials, we define the \textit{quadratic module} of $\vg \in \R[\vx]^m$ as
$$\Sigma(\vg) := \{\bsigma \cdot \vg : \bsigma \in \Sigma[\vx]^m \}.  $$
Then, $\bK := \{\vx \in \R^n : \vg(\vx) \geq \vO\}$ is said to have the \textit{Archimedean property} if
$$ \exists \ R > 0 \ \ \st \ R^2 - \|\cdot\|^2 \in \Sigma(\vg). $$
It is worth noticing, as SOS polynomials are by definition nonnegative, that by definition of $\bK$ one has $\Sigma(\vg) \subset \R[\bK]_+$ and thus the Archimedean property implies $\bK \subset \bB_R$ for some $R > 0$..
\end{rem}

\begin{rem}[Positivstellensatz] \label{rem:psatz}
\begin{itemize}
\item[]
\end{itemize}
The key instrument for the moment-SOS hierarchy is Putinar's Positivstellensatz \cite[Theorem 1.3]{putinar}, which has a dual counterpart under the form of Putinar's Lemma \cite[Lemma 3.2]{putinar}. If the Archimedean property holds, then Putinar's Lemma gives a sufficient hierarchy of LMI conditions for a multi-indexed sequence $\vz = (z_\vk)_{\vk \in \N^n}$ to represent the moments of a measure $\mu \in \cM(\bK_i)_+$, which will be instrumental for the moment hierarchy definition and convergence.

Conversely, under the same assumption, Putinar's Positivstellensatz ensures that $\Sigma(\vg_i)$ is dense in $\cC(\bK_i)_+$ w.r.t. the uniform topology, so that
\begin{alignat*}{3}
\od^\star_\GM := \ && \inf \ & \vx \cdot \va + \vy \cdot \vb 
\\
&& \st \ & \vv + \vw - \vc \in \Sigma(\vg_1)_+ \x \ldots \x \Sigma(\vg_N)_+ \nonumber \\
&&& \vv = \sum_{\alpha \in \A} x_\alpha \, \bvarphi_\alpha \qquad \qquad \vx = (x_\alpha)_{\alpha \in \A} \in \R^\A \ \st \ \{\alpha \in \A : x_\alpha \neq 0\} \text{ is finite} \nonumber \\
&&& \vw = \sum_{\beta \in \B} y_\beta \, \bpsi_\beta \qquad \qquad \vy = (y_\beta)_{\beta \in \B} \in (\R_+)^\B \nonumber 
\end{alignat*}

Then, the isomorphism between sums of squares and positive semidefinite matrices, provided by \cite[Proposition 2.1]{lasserre}, allows to parameterize inequality constraints as LMI constraints.

Note that this means that when restricting the constraints to bounded degree in the hierarchy, only a finite number of constraints need to be checked.
\end{rem}

\section{Notations and theorem statement}

We briefly recall the usual notations and definitions that are used in the Lasserre hierarchy framework. Then, we proceed to introduce the last assumptions that will be needed for the proof of our theorem.

The basic notions that are used to formulate the moment hierarchy are the Riesz functional and the localizing matrix.

\begin{dfn}[Riesz functional]\label{dfn:rieszfun}
\begin{itemize}
\item[]
\end{itemize}
Let $\vz := (z_\vk)_{\vk \in \N^n} \in \R^{\N^n}$ be a real sequence. We define, for $p(\vx) := \sum_{|\vk|\leq d} p_\vk \ \vx^\vk \in \R[\vx]$,
$$ \mL_\vz(p) := \sum_{|\vk|\leq d} p_\vk \ z_\vk, $$
{where for $\vk = (k_1,\ldots,k_n) \in \N^n$, $|\vk| := k_1 + \ldots + k_n$.}

The linear map $\mL_\vz : p \mapsto \mL_\vz(p)$ is called the \emph{Riesz functional} of $\vz$.
\end{dfn}

\begin{rem}[Link between $\mL_\vz$ and integrals] \label{rem:rieszfun}
\begin{itemize}
\item[]
\end{itemize}
If $\vz = (z_\vk)_{\vk\in\N^n} \in \R^{\N^n}$ and $\mu \in \cM(\bK)_+$ are such that
$$ \forall \vk \in \N^n \qquad z_\vk = \int \vx^\vk \; d\mu, $$
then by definition for all $p \in \R[\vx]$
$$ \int p \; d\mu = \mL_\vz(p). $$
\end{rem}

\begin{note}[Bounded multi-indices]
\begin{itemize}
\item[]
\end{itemize}
For $n,d \in \N$, we use the following notations:
$$ \N^n_d := \{\vk \in \N^n : |\vk| \leq d \} \quad {\text{and}} \quad N_{n,d} := {\mathrm{card}(\N^n_d)} = \binom{n+d} {n}, $$
so that the space {$\R_d[\vx]$} of polynomials in $n$ variables with degree at most $d$ {is isomorphic to $\R^{\N^n_d}$ and $\R^{N_{n,d}}$.}
\end{note}

\begin{dfn}[Localizing matrix] \label{dfn:locmat}
\begin{itemize}
\item[]
\end{itemize}
Let $d,d_g \in \N$, $g \in \R_{d_g}[\vx]$. Let $\ve_d(\vx) := (e_i(\vx))_{1 \leq i \leq N_{n,d}}$ be a base of $\R_d[\vx]$. Let $\vz = (z_\vk)_{|\vk| \leq 2d+d_g} \in \R^{\N^n_{2d+d_g}}$.

The \emph{degree $d$ localizing matrix} $\mM_d(g \ \vz)$ 
of $\vz$ in $g$ 
is defined as the size $N_{n,d}$ matrix representation in base $\ve_d(\vx)$ of the bilinear application
$$ (p,q) \in \R_d[\vx]^2 \longmapsto \mL_\vz(g \ p \ q).$$
\end{dfn}

The localizing matrix is defined so that if $p(\vx) = \vp \cdot \ve_d(\vx)$ and $q(\vx) = \vq \cdot \ve_d(\vx)$, $\vp,\vq \in \R^{N_{n,d}}$, then
$$ \mL_\vz(g \ p \ q) = \vp^\top \mM_d(g \ \vz) \ \vq \qquad {\text{and}} \qquad \mL_\vz(g \ p^2) = \vp^\top \mM_d(g \ \vz) \ \vp. $$

\begin{note}[Positive semidefinite matrices]
\begin{itemize}
\item[]
\end{itemize}
We denote by:
\begin{itemize}
\item $\S^n := \{\mM \in \R^{n\x n} : \mM^\top = \mM \}$ the space of symmetric matrices,
\item $\S^n_+ := \{\mM \in \S^n : \forall \vx \in \R^n, \vx^\top \ \mM \ \vx \geq 0 \}$ the closed convex cone of positive semidefinite matrices,
\item if $\mM \in \S^n$, $\mM \succeq 0 \Longleftrightarrow \mM \in \S^n_+$.
\end{itemize}
\end{note}

These definitions allow to formulate the moment hierarchy associated to Problem \eqref{opt:gmp}:

\begin{alignat}{4}
\op^d_\GM := \ && \sup \ & \sum_{i=1}^N \mL_{\vz_i}(c_i) \label{opt:mom} \\
&& \st \ & \ \vz_i \in \R^{\N^{n_i}_{2d}} && \qquad i \in \N^\star_N \nonumber \\
&&& \ \mM_{d-d_{ij}}(g_{i,j} \ \vz_i) \succeq 0 && \qquad i \in \N^\star_N, j \in \N^\star_{m_i} \nonumber \\
&&& \ \sum_{i=1}^N \mL_{\vz_i}(\varphi_{\alpha,i}) = a_\alpha && \qquad \alpha \in \A_d \nonumber\\
&&& \ \sum_{i=1}^N \mL_{\vz_i}(\psi_{\beta,i}) \leq b_\beta && \qquad \beta \in \B_d, \nonumber
\end{alignat}
where $d_{ij} = \lceil {\deg \,} g_{i,j} / 2 \rceil$, and $$\A_d := \{\alpha \in \A : \forall i \in \N^\star_N, {\deg \,} \varphi_{\alpha,i} \leq 2d \},$$ $$\B_d := \{\beta \in \B : \forall i \in \N^\star_N, {\deg \,} \psi_{\beta,i} \leq 2d \}$$
are taken finite (which is possible using linearity of the Riesz functional, {Remark \ref{rem:psatz}} and finite dimensionality of $\R_d[\vx] = \{p\in\R[\vx] : \deg \, p \leq d\}$).

{A vector $\vz_i \in \R^{\N^{n_i}_{2d}}$ that is feasible for problem \eqref{opt:mom} is then called a \emph{pseudo-moment sequence}.
}

{\begin{rem}[SOS hierarchy]
\begin{itemize}
\item[]
\end{itemize}
The dual to the moment hierarchy is called the SOS hierarchy, and is written as follows:
\begin{alignat}{3}
\od^d_\GM := \ && \inf \ & \vx \cdot \va + \vy \cdot \vb \label{opt:sos} \\
&& \st \ & \vv + \vw - \vc \in \Sigma(\vg_1) \x \ldots \x \Sigma(\vg_N) \nonumber \\
&&& \vv = \sum_{\alpha \in \A_d} x_\alpha \, \bvarphi_\alpha \qquad \qquad \vx = (x_\alpha)_{\alpha \in \A_d} \in \R^{\A_d} \nonumber \\
&&& \vw = \sum_{\beta \in \B_d} y_\beta \, \bpsi_\beta \qquad \qquad \vy = (y_\beta)_{\beta \in \B_d} \in (\R_+)^{\B_d}, \nonumber
\end{alignat}
where $\va := (a_\alpha)_{\alpha\in\A_d}$ and $\vb := (b_\beta)_{\beta \in \B_d}$.
\end{rem}
}

\begin{rem}[Minimal degree for the hierarchy]
\begin{itemize}
\item[]
\end{itemize}
It appears in \eqref{opt:mom} that $\op^d_\GM$ is only defined for $d \geq d_0$, where
$$ d_0 := \max_{i \in \{1,\ldots,N\}} \max_{j \in \{1,\ldots,m_i\}} \max (\lceil {\deg \,} c_i / 2 \rceil , d_{ij}). $$
\end{rem}

We are finally able to formulate the last needed assumptions as well as our main theorem. First, convergence in the moment hierarchy is often obtained through an additional assumption on the mass of the involved measures:

\begin{asm}[Uniformly bounded mass] \label{asm:bound}
\begin{itemize}
\item[]
\end{itemize}
For $i \in \{1,\ldots,N\}$, denote by $\vz_i \in \R^{\left(\N^{n_i}_{2d}\right)}$ the pseudo-moment sequence that represents $\mu_i$ in the Lasserre hierarchy, and suppose that there exists $C_i > 0$ s.t. if $\vz_i$ is feasible for the degree $d$ relaxation of problem \eqref{opt:gmp}, then
$$ z_{i,\vO} \leq {C_i}. $$
\end{asm}

Eventually, {our proof of convergence consists in proving that our hierarchy of pseudo-moment sequences has at least one accumulation point.} So as {to deduce actual convergence}, we need a unique candidate for {such accumulation point}:

\begin{asm}[Existence of a unique optimal solution] \label{asm:unique}
\begin{itemize}
\item[]
\end{itemize}
Suppose that there exists a unique $\bmu^\star \in \cM(\bK_1)_+\x\ldots\x\cM(\bK_N)_+$ feasible and optimal for problem \eqref{opt:gmp}:
$$ \begin{cases} \forall \alpha \in \A, \int \bvarphi_\alpha \cdot d\bmu^\star = a_\alpha \\ \forall \beta \in \B, \int \bpsi_\beta \cdot d\bmu^\star \leq b_\beta
\end{cases} \qquad {\text{and}} \qquad \int \vc \cdot d\bmu^\star = \op^\star_\GM. $$
\end{asm}

\begin{rem}[The question of uniqueness]
\begin{itemize}
\item[]
\end{itemize}
Assumption \ref{asm:unique} is crucial for the proof of convergence. However, if one removes it, \emph{existence} of the optimal $\bmu^\star$ could still be obtained from Assumption \ref{asm:bound}, through an infinite dimensional strong duality proof similar to what we present in the next section (which would bring no additional theoretical insight, so that we do not display it). The actual stake of this assumption is indeed the \emph{uniqueness} assumption. Note that in most of the moment-SOS hierarchy applications, the GMP is designed in a way that enforces Assumption \ref{asm:unique}.
\end{rem}

Our contribution consists in stating and proving a general theorem for convergence of the Lasserre hierarchy \eqref{opt:mom} corresponding to problem \eqref{opt:gmp}:

\begin{thm}[Convergence of the general Lasserre hierarchy] \label{thm:main}
\begin{itemize}
\item[]
\end{itemize}
Under Assumptions \ref{asm:ball},  \ref{asm:bound} and \ref{asm:unique}, there exist sequences $\left(\vz_{i}^d\right)_{d \geq d_0}$ of feasible pseudo-moment sequences for the moment hierarchy {\eqref{opt:mom}} s.t. $\sum_{i=1}^N\mL_{\vz^d_i}(c_i) = \op^d_\GM${. Then,} for all $\vk \in \N^n$,
$$ z^d_{i,\vk} \underset{d\to\infty}{\longrightarrow} \int \vx^\vk \; d\mu_i^\star(\vx). $$ 
In particular, one has $\op^d_\GM \underset{d\to\infty}{\longrightarrow} \op^\star_\GM$.

Moreover, this automatically yields strong duality $\op^d_\GM = \od^d_\GM$ {and} $\op^\star_\GM = \od^\star_\GM$, where $\od^\star_\GM$ and $\od^d_\GM$ are the values of the duals {\eqref{opt:gmpd} and \eqref{opt:sos}}, respectively, so that $\od^d_\GM \underset{d\to\infty}{\longrightarrow} \od^\star_\GM$.
\end{thm}

{
\begin{rem}[Contribution w.r.t. \cite{joszdual} and \cite{lasserre}]
\begin{itemize}
\item[]
\end{itemize}
Whereas \cite{joszdual} and \cite[Theorem 5.6(b)]{lasserre} respectively prove strong duality in the moment-SOS hierarchy and convergence of the resulting pseudo-moments towards the unique solution to the original POP, this contribution extends these results from the particular POP problem to \emph{all} possible instances of the GMP, including the various examples mentioned in introduction, which should facilitate future contributions in the field of moment-SOS hierarchies.

Such generality is obtained by enforcing, in addition to the standard Assumptions \ref{asm:ball} and \ref{asm:unique} that were already introduced in \cite{joszdual} and \cite[Theorem 5.6(b)]{lasserre}, of the boundedness Assumption \ref{asm:bound} that guarantees strong duality in the GMP, a property that is trivially established in the context of POP but can become difficult to check in other situations, such as \cite{reachable} and \cite{inmpi}.
\end{rem}
}

\section{Proof of Theorem \ref{thm:main}}

For the sake of simplicity and without loss of generality,
we work on the simple case \eqref{opt:gmpo} where $N = 1$, with { $\B = \varnothing$\footnote{Indeed as $\B$ is finite, there exists a finite value of $d$ that captures all inequality constraints; arbitrarily setting this value to $0$ does not change our proof for convergence.}, whose} corresponding hierarchy { writes}:

\begin{alignat}{4}
\op^d_\GM := \ && \sup \ & \mL_\vz(c) \label{opt:momN} \\
&& \st \ & \vz \in \R^{\N^n_{2d}} \nonumber \\
&&& \mM_{d-d_j}(g_j \ \vz) \succeq 0 && \qquad j \in \N^\star_m \nonumber \\
&&& \mL_\vz(\varphi_\alpha) = a_\alpha && \qquad \alpha \in \A_d. \nonumber
\end{alignat}

All the proofs are the same for $N > 1$, at the price of additional notations that do not bring any theoretical insight, so that we stick to the case $N=1$.
We first prove an easy lemma that gives all its importance to Assumption \ref{asm:ball}:

\begin{lem}[Pseudo moment sequences boundedness] \label{lem:locmat}
\begin{itemize}
\item[]
\end{itemize}
Let $d \in \N{\setminus\{0\}}$, $R>0$, $\vz \in \R^{\N^n_{2d}}$ s.t. $\mM_d(\vz)\succeq0$ {and} $\mM_{d-1}((R^2 - {\|}\cdot{\|}^2)\ \vz) \succeq 0$. Then,
$$ { \min_{|\vk| \leq d} z_{2\vk} \geq 0 \quad \text{and} \quad} \max_{|\vk|\leq 2d}|z_\vk| \leq z_\vO \ \max(1,R^{2d}). $$
\end{lem}
\begin{proof}
$\mM_d(\vz) \succeq 0$ is equivalent to
\begin{equation} \label{eq:mom} \tag{a}
\forall p \in \R_d[\vx], \mL_\vz(p^2) \geq 0,
\end{equation}
while
$\mM_{d-1}((R^2 - {\|}\cdot{\|}^2)\ \vz) \succeq 0$ means that
\begin{equation} \label{eq:loc} \tag{b}
\forall p \in \R_{d-1}[\vx], \mL_\vz((R^2 - {\|}\cdot{\|}^2) \ p^2) \geq 0.
\end{equation}
\eqref{eq:mom} with $p(\vx) = \vx^\vk$, $|\vk| \leq d$ yields $z_{2\vk} \geq 0$.

\eqref{eq:loc} with $p(\vx)=1$ yields $R^2z_\vO \geq \sum_{|\vk|=1} z_{2\vk}$, since ${\|}\vx{\|}^2 =\sum_{j=1}^n x_j^2 = \sum_{|\vk|=1} \vx^{2\vk}$.
Hence, since the  $z_{2\vk}$ are nonnegative, one has $|\vk|=1 \Rightarrow z_{2\vk} \leq R^2z_\vO$.

Going forward, if $|\vk|=1$, \eqref{eq:loc} with $p(\vx)=\vx^{\vk}$ yields $R^2z_{2\vk} \geq \sum_{|\vk'|=1} z_{2(\vk+\vk')}$
with $z_{2(\vk+\vk')} \geq 0$ by \eqref{eq:mom}, so that $R^4 z_\vO \geq R^2z_{2\vk} \geq z_{2(\vk+\vk')}$ as long as $|\vk| = |\vk'| = 1$, and thus, if $|\vk|=2$, $R^4 z_\vO \geq z_{2\vk}$. By induction, one has for 
$\vk\in\N^n_d$ that
\begin{equation}\label{eq:even}
0 \leq z_{2\vk} \leq R^{2|\vk|}z_\vO \leq z_\vO \ \max(1,R^{2d}). \tag{c}
\end{equation}
Let $\vk,\vk' \in \N^n_{d}$. Then, \eqref{eq:mom} with $p(\vx) = \vx^\vk \pm \vx^{\vk'}$ yields $0 \leq \mL_\vz(p^2) = z_{2\vk} \pm 2z_{\vk+\vk'} + z_{2\vk'}$ so that
\begin{equation} \label{eq:odd}
|z_{\vk+\vk'}| \leq \frac{z_{2\vk} + z_{2\vk'}}{2} \leq \max(z_{2\vk},z_{2\vk'}) \stackrel{\eqref{eq:even}}{\leq} z_\vO \ \max(1,R^{2d}). \tag{d}
\end{equation}
Eventually, any $\vk \in \N^n_{2d}$ can be written $\vk = \vk'+\vk''$ with $\vk',\vk''\in\N^n_d$, so that it satisfies \eqref{eq:odd}, and $|z_\vk| \leq z_\vO \ \max(1,R^{2d})$.
\end{proof}
This lemma proves several important facts, among which any nonzero feasible pseudo-moment vector $\vz_i$ for \eqref{opt:mom} satisfies $z_{i,\vO} > 0$ if Assumption \ref{asm:ball} holds, and under Assumptions \ref{asm:ball} and \ref{asm:bound}, for all $d \in \N$, $\vz_i$ feasible for \eqref{opt:mom} satisfies $|z_{i,\vk}| \leq C_i$ $\forall \vk$.

\begin{rem}[Relaxing Assumption \ref{asm:ball}]
\begin{itemize}
\item[]
\end{itemize}
Actually, this lemma still holds if one replaces Assumption  \ref{asm:ball} with the Archimedean property. Indeed, suppose that $R^2 - \|\cdot\|^2 \in \Sigma(\vg)$ for some $R>0$. Then, there exists $\bsigma \in \Sigma[\vx]^m$ such that 
$$R^2 - \|\cdot\|^2 = \bsigma \cdot \vg = \sum_{i=1}^m \sigma_i \ g_i = \sum_{i=1}^m \sum_{j=1}^{k_i} p_{i,j}^2 \ g_i.$$
Then, let $D := \max_{i,j} \lceil \deg (p_{i,j}^2)/2 \rceil$ and for $i \in \N_m^\star$ let $d_i := d + D - 1$ and suppose $\mM_{d_i}(g_i \, \vz) \succeq 0$. In that case, for $q \in \R_{d-1}[\vx]$, one recovers \eqref{eq:loc}:
$$ \mL_\vz((R^2 - \|\cdot\|^2) \ q^2) =  \sum_{i=1}^m \sum_{j=1}^{k_i} \mL_\vz(g_i \ p_{i,j}^2 \ q^2) \geq 0. $$
However, we chose to keep Assumption \ref{asm:ball} instead of the Archimedean property, as the former is much easier to verify in practice and can always be enforced when the latter holds, at the price of adding a redundant ball constraint to the $\bK_i$'s descriptions (see remarks \ref{rem:ball} and \ref{rem:archi}).
\end{rem}

A first consequence of Lemma \ref{lem:locmat} is the strong duality property in the hierarchy.

\begin{prop}[Strong duality in the hierarchy] \label{prop:strongdual}
\begin{itemize}
\item[]
\end{itemize}
Suppose that each relaxation of Problem \eqref{opt:gmpo} has a feasible solution. In that case, under Assumptions \ref{asm:ball} and \ref{asm:bound},
$$ \forall d \geq d_0, \quad \op^d_\GM = \od^d_\GM. $$
Moreover, 
the degree $d$ relaxation \eqref{opt:momN} of \eqref{opt:gmpo} has an optimal solution $\vz_d \in \R^{\N^n_{2d}}$.
\end{prop}
\begin{proof}
We rely on \cite[{Chapter IV:} Theorem (7.2), Lemma (7.3)]{barvinok} to prove our result. Consider the cone 
$$ \cK_d := \left\{ \left(\mA_d \vz , \mL_\vz(c)\right) : \vz \in \cX_d\right\}, $$
where $\mA_d \vz := (
\mL_\vz(\varphi_\alpha))_{
\alpha \in \A_d}$ and $$\cX_d := \left\{\vz\in\R^{\N^n_{2d}} : \forall j \in \{1,\ldots,m\}, \mM_{d-d_j}(g_j \ \vz) \succeq 0 \right\}.$$

According to \cite[Theorem (7.2)]{barvinok} a sufficient condition for our result to hold requires that $\op^d_\GM < \infty$ and $\cK_d$ is closed. Clearly,
$$\op^d_\GM \leq N_{n,d_0} \left(\max_{\vk\in\N^n} c_\vk\right) \left(\max_{\vz}\max_{\vk\in\N^n_{2d}} z_\vk\right) \leq N_{n,d_0} \left( \max_{\vk\in\N^n} c_\vk \right) C < \infty.$$
Besides, \cite[Lemma (7.3)]{barvinok} states that for $\cK_d$ to be closed, it is sufficient to prove that $\cX_d$ has a compact, convex base, and that
\begin{equation} \label{eq:ker} \tag{$*$}
\forall \vz \in \cX_d, \qquad \left(\mA_d \vz, \mL_\vz(c) \right) = (\vO,0) \Longrightarrow \vz = \vO.
\end{equation}
We first exhibit a compact convex base for $\cX_d$. {A base for a cone is defined as follows:
\begin{dfn}[Convex base of a cone]
Let $\cX$ be a cone, \emph{i.e.} a subset of a real vector space that is invariant under multiplication by nonnegative real numbers. $\cP \subset \cX$ is said to be a \emph{base} of $\cX$ if the map
$$\chi : \left\{\begin{array}{ccc}
(0,+\infty)\x\cP & \longrightarrow & \cX\setminus\{\vO\} \\
(t,\vp) & \longmapsto & t\,\vp
\end{array}\right. $$
is a bijection.
\end{dfn}
We are looking for such a base, with the additional properties that it should be convex (\emph{i.e.} stable by barycenter operation) and compact (\emph{i.e.}, in this finite dimensional context, closed and bounded).
}
Let
$$ \cP_d := \{\vz = (z_\vk)_\vk \in \cX_d : z_\vO = 1 \}. $$

$\cP_d$ is a base of $\cX_d$ in the sense that $\cX_d \setminus \{\vO\}$ is isomorphic to $(0,+\infty)\x\cP_d$ through the bijective application $\chi_d : (t,\vp) \mapsto t \ \vp$, with $\chi_d^{-1}(\vz) = (z_\vO , \vz/z_\vO)$ (using Lemma \ref{lem:locmat} and Assumption \ref{asm:ball}, for any $\mM_d(\vz)$ and $\mM_{d-1}((1 - {\|}\cdot{\|}^2) \ \vz)$ to be simultaneously positive semi-definite with $\vz \neq \vO$, it is necessary that $z_\vO > 0$).

$\cP_d$ is convex. Indeed, let $\vp_1,\vp_2 \in \cP_d$, $t \in [0,1]$, $\widetilde{\vp} := t \ \vp_1 + (1-t) \ \vp_2$. Then, by linearity of the localizing matrix operator,
$$ \forall g \in \R_d[\vx] \ \mM_{d-d_g}(g \ \widetilde{\vp}) = t \ \mM_{d-d_g}(g \ \vp_1) + (1-t) \ \mM_{d-d_g}(g \ \vp_2), $$
where $d_g = \lceil {\deg \,} g / 2 \rceil$, so that its semidefinite positivity is preserved by convex combination, by convexity of $\S^n_+$ for any $n \in \N$. Thus, $\widetilde{\vp} \in \cX_d$. Besides,
$$ \widetilde{p}_\vO = t \ p_{1\vO} + (1-t) \ p_{2\vO} = t + 1 - t = 1 $$
so that $\widetilde{\vp} \in \cP_d$, which proves convexity.

Eventually, we move on to showing compactness of $\cP_d$. According to Lemma \ref{lem:locmat}, the ball constraint in the description of $\bK$ (Assumption \ref{asm:ball}) and the upper bound $z_\vO \leq C$ (Assumption \ref{asm:bound}) yield boundedness of $\cP_d$ for the distance $\mathrm{dist}(\vz,\vz') := \max_{|\vk|\leq2d} |z_\vk - z'_\vk|$.

Finally, $\cP_d$ is closed as the intersection between the level-$1$ set of the continuous function $\vz \mapsto z_\vO$ and the closed cone $\cX_d$. Indeed $\cX_d$ is closed as the pre-image of the closed cone $\left(\S^{N_{n,d}}_+\right)^{m}$ by the (continuous) linear map
$$\vz \mapsto \left(\mM_{d-d_j}(g_j \ \vz)\right)_{j\in\{1,\ldots,m\}}.$$
Since finite dimensional closed bounded sets are compact, this proves that $\cP_d$ is compact.

It remains to prove \eqref{eq:ker}. Let $\vz \in \cX_d$ s.t. $\mA_d \vz = \vO$ and $\mL_{\vz}(c) = 0$. We want to prove that $\vz = \vO$ so that \eqref{eq:ker} holds.

Let $\vz_0 \in \cX_d$ s.t. $\forall 
\alpha \in \A_d$, 
$\mL_{\vz_0}(\varphi_\alpha) = a_\alpha$. Define for $t \geq 0$ $\vz_t := \vz_0 + t \ \vz$. Let $t \geq 0$.

Since $\cX_d$ is a convex cone, $\vz_t \in \cX_d$. In addition, by construction of $\vz$ and linearity of the operator $\vz \mapsto \mL_\vz(\cdot)$, $\forall  
\alpha \in \A_d$, 
$\mL_{\vz_t}(\varphi_\alpha) = \mL_{\vz_0}(\varphi_\alpha) = a_\alpha$, so that Assumption \ref{asm:bound} ensures that $z_{t,\vO} \leq C$. However, $z_{t,\vO} = z_{0,\vO} + t \ z_\vO$.

Combined with the fact that $z_{0,\vO} \geq 0$ (with Lemma \ref{lem:locmat}), this yields that for all $t \geq 0$,
$$ t \ z_\vO \leq C, $$
which is only possible if $z_\vO=0$, \emph{i.e.} if $\vz = \vO$ using again Lemma \ref{lem:locmat}.
\end{proof}

Eventually, we can end this section by proving Theorem \ref{thm:main}. As announced at the beginning of this section, w.l.o.g we actually only do the proof for the case $N = 1$:

\begin{thm}[Convergence of the pseudo-moment sequences] \label{thm:lasconv}
\begin{itemize}
\item[]
\end{itemize}
Under Assumptions \ref{asm:ball},  \ref{asm:bound} and \ref{asm:unique}, there exists a sequence $(\vz_d)_{d \geq d_0}$ of feasible pseudo-moment sequences for the hierarchy \eqref{opt:momN} corresponding to Problem \eqref{opt:gmpo}, s.t. $\mL_{\vz_d}(c) = \op^d_\GM${. Then,} for all $\vk \in \N^n$,
$$ z_{d,\vk} \underset{d\to\infty}{\longrightarrow} \int \vx^\vk \; d\mu^\star(\vx). $$ 
In particular, one has $\op^d_\GM \underset{d\to\infty}{\longrightarrow} \op^\star_\GM$.

Moreover, this automatically yields strong duality $\op^d_\GM = \od^d_\GM$ {and} $\op^\star_\GM = \od^\star_\GM$ and then dual convergence $\od^d_\GM \underset{d\to\infty}{\longrightarrow} \od^\star_\GM$.
\end{thm}
\begin{proof}
Existence of $(\vz_d)_{d\geq d_0}$ follows from Proposition \ref{prop:strongdual} (using Assumptions \ref{asm:ball}
 and \ref{asm:bound}), so we focus on the proof of convergence. Let $d \in \N$. For $\vk \in \N^n$, define
$$ \hat{z}_{d,\vk} := \left\{\begin{array}{ll}
z_{d,\vk} & \qquad \text{ if } |\vk|\leq 2d, \\ 0 & \qquad \text{ else},
\end{array}\right.$$
so that $\hat{\vz}_d \in \R^{\N^n}$ with
$${\|}\hat{\vz}_d{\|}_{\ell^\infty(\N^n)} := \max_{\vk\in\N^n} |\hat{z}_{d,\vk}| = \max_{|\vk|\leq 2d} |z_{d,\vk}| \stackrel{\text{Lemma \ref{lem:locmat}}}{\leq} z_{d,\vO} \leq C,$$
using again Assumptions \ref{asm:ball} and \ref{asm:bound} {(recall that in particular Assumption \ref{asm:ball} sets $R = 1$ in the application of Lemma \ref{lem:locmat})}.
Then, $(\hat{\vz}_d)_{d\in\N}$ is a uniformly bounded sequence of
$$\ell^\infty(\N^n) := \left\{\vu \in \R^{\N^n} : \max_{\vk\in\N^n} |u_\vk| < \infty \right\},$$
{which is the topological dual of $\ell^1(\N^n) := \left\{\vu \in \R^{\N^n} : \sum_{\vk\in\N^n} |u_\vk| < \infty \right\}$.}
Thus, the Banach-Alaoglu theorem \cite[Theorem 3.16]{brezis} yields a weak-$\ast$ converging subsequence $(\hat{z}_{d_r})_{r\in\N}$: $\exists \, \vz_\infty \in \ell^\infty(\N^n) \ \st \ \forall \vu \in \ell^1(\N^n)$,
$$ \sum_{\vk\in\N^n} u_\vk \ z_{d_r,\vk} \underset{r\to\infty}{\longrightarrow} \sum_{\vk\in\N^n} u_\vk \ z_{\infty,\vk}.$$
In particular, if $\vk \in \N^n$, {$u_\vl := 0$ if $\vl \neq \vk$, $u_\vk = 1$ defines a $\vu \in \ell^1(\N^n)$ for which the above statement becomes} $z_{d_r,\vk} \underset{r\to\infty}{\longrightarrow} z_{\infty,\vk}$.

Thus, what we want to show is that for $\vk \in \N^n$, $z_{\infty,\vk} = \int \vx^\vk \; d\mu^\star(\vx)$ {$=:z^\star_\vk$, where we recall that $\mu^\star$ is the unique solution to \eqref{opt:gmpo} given by Assumption \ref{asm:unique}}.

Let $
\alpha \in \A, j\in\N^\star_m$. Then, for all $d\geq d_0$ and $r \in \N$ large enough, by feasibility of $\vz_{d_r}$ for the relaxation of degree $d_r$, one has

$\begin{array}{ccclclcl}
\bullet & 0 & \preceq & \mM_{d-d_j}(g_j \ \vz_{d_r}) & = & \mM_{d-d_j}(g_j \ \hat{\vz}_{d_r}) & \underset{r\to\infty}{\longrightarrow} & \mM_{d-d_j}(g_j \ \vz_{\infty}) \quad \\
\\
\bullet & a_\alpha & = & \mL_{\vz_{d_r}}(\varphi_\alpha) & = & \mL_{\hat{\vz}_{d_r}}(\varphi_\alpha) & \underset{r\to\infty}{\longrightarrow} & \mL_{\vz_{\infty}}(\varphi_\alpha)
\end{array}$

so that according to Putinar's Lemma {\cite[Lemma 3.2]{putinar}, whose conditions hold by Assumption \ref{asm:ball}}, $\vz_\infty$ is the actual moment sequence of a measure $\mu_\infty$ that is feasible for problem \eqref{opt:gmpo}. Then, one directly has
$$\op^\star_\GM \geq \int c \; d\mu_\infty = \mL_{z_\infty}(c) = \lim_{r\to\infty} \mL_{\hat{z}_{d_r}}(c) = \lim_{r\to\infty} \op^{d_r}_\GM \geq \op^\star_\GM$$
since for any $d \geq d_0$ $\op^d_\GM \geq \op^\star_\GM$ (by construction of the moment \textit{relaxation}).

Hence, $\int c \; d\mu_\infty = \op^\star_\GM$, \emph{i.e.} $\mu_\infty$ is \emph{optimal} for problem \eqref{opt:gmpo}. By Assumption \ref{asm:unique}, this yields $\mu_\infty = \mu^\star$, \emph{i.e.} $\vz_\infty = \vz^\star$. Thus, $(\hat{\vz}_d)_d$ is bounded and has a unique weak-$\ast$ accumulation point $\vz^\star$, which means that for any $\vk \in \N^n$,
$$ z_{d,\vk} \underset{d\to\infty}{\longrightarrow} z^\star_\vk = \int \vx^\vk \; d\mu^\star(\vx). $$

Eventually, since Assumptions \ref{asm:ball} and \ref{asm:bound} hold, Proposition \ref{prop:strongdual} ensures strong duality $\op^d_\GM = \od^d_\GM$, so that putting together weak GMP duality and the strenghtening property, one has
$$ \op^\star_\GM \stackrel{\text{weak duality}}{\leq} \od^\star_\GM \stackrel{\text{strenghtening}}{\leq} \od^d_\GM \stackrel{\text{strong duality}}{=\vphantom{\leq}} \op^d_\GM \stackrel{\text{convergence}}{\underset{d\to\infty}{\longrightarrow}\vphantom{\leq}} \op^\star_\GM $$
and the sandwich rule yields strong GMP duality $\op^\star_\GM = \od^\star_\GM$.
\end{proof}

\section{{Discussing the uniqueness} Assumption \ref{asm:unique}}

{In practice, it is quite easy to enforce Assumptions \ref{asm:ball} and \ref{asm:bound}, up to rescaling and addition of mass constraints. However, Assumption \ref{asm:unique} may not always hold (especially in the case of polynomial optimization or PDE solutions), so that it is worth discussing what can be done without this Assumption.}
If Assumption \ref{asm:unique} does not hold, then our contribution reduces to the following result:

\begin{cor}[to Theorem \ref{thm:lasconv}] \label{cor:main}
\begin{itemize}
\item[]
\end{itemize}
Under Assumptions \ref{asm:ball} and \ref{asm:bound}, there exists a \emph{subsequence} $(\vz_{d_r})_{r\in\N}$, where $(d_r)_{r\in\N} \in ([d_0,\infty) \cap \N)^\N$ is strictly increasing, of feasible pseudo-moment sequences for the hierarchy \eqref{opt:momN}, as well as a measure $\mu_\infty \in \cM(\bK)_+$ feasible for Problem \eqref{opt:gmpo}, s.t. $\mL_{\vz_{d_r}}(c) = \op^{d_r}_{\GM}$, $\int c \; d\mu_\infty = \op^\star_{\GM}$ and for all $\vk \in \N^n$,
$$ z_{d_r,\vk} \underset{r\to\infty}{\longrightarrow} \int \vx^\vk \; d\mu_\infty(\vx). $$
In particular, one still has $\op^d_\GM \underset{d\to\infty}{\longrightarrow} \op^\star_\GM$.

Moreover, this automatically yields strong duality $\op^d_\GM = \od^d_\GM$ {and} $\op^\star_\GM = \od^\star_\GM$ and then dual convergence $\od^d_\GM \underset{d\to\infty}{\longrightarrow} \od^\star_\GM$.
\end{cor}
\begin{proof}
Going back to the proof of Theorem \ref{thm:lasconv}, most arguments are still valid, which lets us with a \emph{subsequence} $(\vz_{d_r})_{r \in \N}$ that converges towards the moments of an \emph{optimal} $\mu_\infty$. The only part that is not valid anymore without Assumption \ref{asm:unique} is the identification between $\mu_\infty$ and $\mu^\star$, and thus the convergence of the whole sequence $(\vz_d)_{d \geq d_0}$. Indeed, in such context, the pseudo-moment sequence might have several accumulation points $\mu^\star_1,\ldots,\mu^\star_M$ (or even an infinity of accumulation points), and either converge to one of them, or oscillate between them. The strong duality and value convergence results being independent from this, the rest of Theorem \ref{thm:lasconv} is unchanged.
\end{proof}

{To illustrate this result}, consider the POP problem $f^\star := \inf_{x \in [0,1]} (1-x)x$, whose value is obviously $f^\star = 0$ with minimizers $x^\star=0$ and $x^\star=1$. We then write the corresponding GMP instance as well as its order $d$ moment relaxation:

\begin{subequations} \label{eq:pop}
\begin{center}
\begin{minipage}{0.45\textwidth}
\begin{alignat}{4}
\op^\star_f := \ && \inf \ & \int (1-x)x \; d\mu(x) \label{opt:pop} \\
&& \st \ & \ \ \ \mu \in \cM([0,1])_+ \nonumber \\
&&&\int 1 \; d\mu(x) = 1, \nonumber \\
\nonumber \\
\nonumber
\end{alignat}
\end{minipage}
\begin{minipage}{0.45\textwidth}
\begin{alignat}{4}
\op^d_f := \ && \ \inf \ & z_1 - z_2 \label{opt:popm} \\
&& \st \ & \ \vz = (z_k)_{0\leq k \leq 2d} \in \R^{\N_{2d}} \nonumber \\
&&& \mM_d(\vz) \succeq 0 \nonumber \\
&&& \mM_{d-1}\left(x(1-x)\vz\right) \succeq 0 \nonumber \\
&&& \mM_{d-1}\left((1 - x^2)\vz\right) \succeq 0 \nonumber \\
&&& z_0 = 1, \nonumber
\end{alignat}
\end{minipage}
\end{center}
\end{subequations}
where $[0,1]$ is described as $[0,1] = \{x \in \R : 1 \geq 0, \ (1-x)x \geq 0 \text{ and } 1-x^2 \geq 0 \}$ so as for Assumption \ref{asm:ball} to hold. Note that constraint $\int 1 \; d\mu = 1$ automatically enforces Assumption \ref{asm:bound} through its moment writing $z_0 = 1$. Now, the obvious solution to problem \eqref{opt:pop} is $\op^\star_f = 0$ with minimizing set $$\bM^\star := \{\mu^\star = t \, \delta_1 + (1-t) \, \delta_0 : t \in [0,1]\}.$$

Indeed, $\bM^\star$ parameterizes all possible probability measures supported on the set $\{0,1\}$ of minimizers of $(1-x)x$ (see figure \ref{fig:graph}). Thus, in this case, Assumption \ref{asm:unique} does not hold.

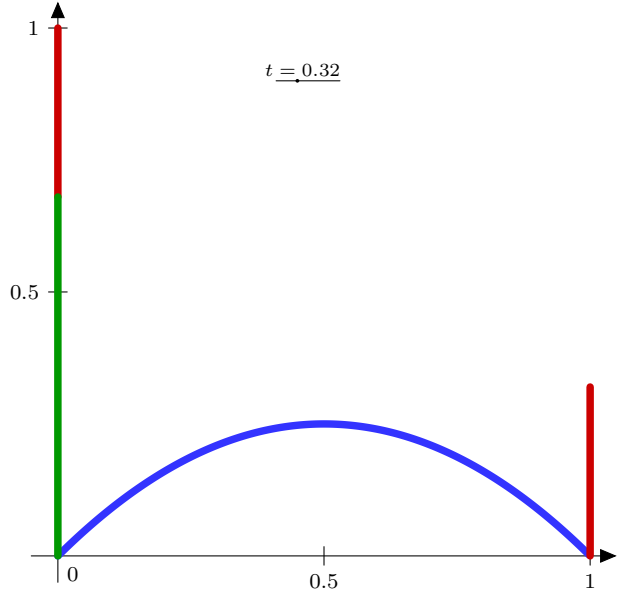
\begin{figure}[h!]

\begin{center}
\begin{tikzpicture}[scale=7,line cap=round,line join=round,>=triangle 45,x=1.0cm,y=1.0cm]
\draw[->,color=black] (-0.05,0) -- (1.05,0);
\foreach \x in {0.5,1}
\draw[shift={(\x,0)},color=black] (0pt,0.5pt) -- (0pt,-0.5pt) node[below] {\footnotesize $\x$};
\draw[->,color=black] (0,-0.05) -- (0,1.05);
\foreach \y in {0.5,1}
\draw[shift={(0,\y)},color=black] (0.5pt,0pt) -- (-0.5pt,0pt) node[left] {\footnotesize $\y$};
\draw[color=black] (0pt,-1pt) node[right] {\footnotesize $0$};
\clip(-0.05,-0.05) rectangle (1.05,1.05);
\draw(0.41,0.9) -- (0.53,0.9);
\draw[line width=2.8pt,color=ttttff, smooth,samples=100,domain=0:1] plot(\x,{(\x)*(1-(\x))});
\draw [line width=2.8pt,color=ccqqqq] (0,1)-- (0,0.68);
\draw [line width=2.8pt,color=qqzzqq] (0,0.68)-- (0,0);
\draw [line width=2.8pt,color=ccqqqq] (1,0.32)-- (1,0);
\begin{scriptsize}
\fill [color=black] (0.45,0.9) circle (0.1pt);
\draw[color=black] (0.46,0.92) node {$t = 0.32$};
\end{scriptsize}
\end{tikzpicture}
\caption{Graph of the function $x \mapsto (1-x)x$ that we minimize.} \label{fig:graph}
\end{center}
We represent a possible optimal measure $\mu^\star = t \, \delta_1 + (1-t) \, \delta_0$ where $t=0.32$ is the length of the red segment and $(1-t) = 0.68$ is the length of the green one.
\end{figure}

{We then proceed to study the moment relaxation \eqref{opt:popm}.

\begin{prop}[Finite convergence]
\begin{itemize}
\item[]
\end{itemize}
For $d \geq 1$, the minimizers of problem \eqref{opt:popm} are in correspondance with the minimizers of problem \eqref{opt:pop}:
$$\bM^d = \left\{ \left( \int x^k \; d\mu^\star(x) \right)_{k \in \N_{2d}} : \mu^\star \in \bM^\star \right\}.$$
This is called a \emph{finite convergence} phenomenon, and it is very common with POP moment hierarchies \cite{finiteconv}.
\end{prop}
}
\begin{proof}
{Let} $d \geq 1$. Let $\vz = (z_k)_{0\leq k\leq 2d} \in \R^{\N_{2d}}$ be feasible for \eqref{opt:popm}. For $k \in \N_{2d-2}$, we define
$$ u_k := z_{k+1} - z_{k+2}, $$
which yields a new {vector} $\vu = (u_k)_{0 \leq k \leq 2d-2}{ \in \R^{2d-1}}$.

Then, $\mM_{d-1}(\vu) = \mM_{d-1}\left(x(1-x)\vz\right) \succeq 0$. We next prove that $\mM_{d-2}\left((1-x^2)\vu\right) \succeq 0$. Let $p \in \R_{d-2}[\vx]$.
\begin{align*}
\mL_\vu\left((1-x^2)p^2\right) & = \mL_\vz\left((1-x^2)p^2(1-x)x\right) \\
& = \mL_\vz\left((1+x)(1-x)^2p^2x\right) \\
& = \mL_\vz\left((1-x+2x)(1-x)^2p^2x\right) \\
& = \underset{\geq 0 \text{ since } \mM_{d-1}(x(1-x)\vz) \succeq 0}{\underbrace{\mL_\vz\left(x(1-x)(1-x)^2p^2\right)}} + 2 \ \underset{\geq 0 \text{ since } \mM_d(\vz) \succeq 0}{\underbrace{\mL_\vz\left(x^2(1-x)^2p^2\right)}} \geq 0,
\end{align*}
which is the definition of $\mM_{d-2}\left((1-x^2)\vu\right) \succeq 0$. Then, applying Lemma \ref{lem:locmat} to $\vu$, we deduce that
\begin{equation} \label{eq:proof}
 0 \leq \max_{1\leq k \leq 2d-1}|z_k - z_{k+1}| = \max_{0 \leq k \leq 2d-2}|u_k| \leq u_0 = z_1 - z_2.
\end{equation}
This proves that $\op^d_f \geq 0$. On the other hand, $\vz = (1,t,\ldots,t) \in \R^{\N_{2d}}$, with $t \in [0,1]$, defines a feasible $\vz$ (corresponding to the moments of optimal measure $\mu^\star = t \, \delta_1 + (1-t) \, \delta_0$) such that $z_1-z_2 = 0$, which proves that $\op^d_f \leq 0$ and thus $\op^d_f = 0$.

Eventually, let $\vz = (z_k)_{0\leq k \leq 2d}$ be feasible \emph{and optimal} for \eqref{opt:popm}: $z_1 - z_2 = 0$. Then, equation \eqref{eq:proof} instantly yields that for all $k \geq 2$, $z_k = z_1$. Moreover, according to Lemma \ref{lem:locmat} applied to $\vz$, for all $k \geq 1$, $z_k = z_2 \in [0,1]$. In other words, for $d \geq 1$ the set of minimizers for problem \eqref{opt:popm} is exactly the set of truncated moment sequences of minimizers for problem \eqref{opt:pop}:
$$ \bM^d = \{(z_k)_{0\leq k\leq 2d} : z_0 = 1, z_1 \in [0,1], \forall k \geq 1, z_k = z_1\} {= \left\{ \left( \int x^k \; d\mu^\star(x) \right)_{k \in \N_{2d}} : \mu^\star \in \mM^\star \right\}}, $$
no matter the size of $d \geq 1$. 
\end{proof}

\begin{rem}[Notions of convergence]
\begin{itemize}
\item[]
\end{itemize}
Here it is important to distinguish between two different notions of convergence:
\begin{itemize}
\item \emph{Finite convergence} of the hierarchy denotes the aforementioned phenomenon, when the minimizers of a finite degree relaxation all correspond to minimizers of the original GMP;
\item One talks about \emph{pseudo-moment sequence convergence} when any sequence $(\vz_d)_d$ of minimizers for \eqref{opt:momN} converges to the sequence of moments of a minimizer for \eqref{opt:gmpo}.
\end{itemize}
Theorem \ref{thm:main} ensures pseudo-moment sequence convergence under Assumptions \ref{asm:ball}--\ref{asm:unique}. Corollary \ref{cor:main} states that removing Assumption \ref{asm:unique} leads to having pseudo-moment sequence convergence \emph{only up to a subsequence $(\vz_{d_r})_r$.}
\end{rem}

{In our illustrative example one indeed has finite convergence, but} there is no \emph{a priori} reason that each step of the hierarchy will return an optimal pseudo-moment sequence corresponding to the same optimal measure, and one could have for example that for $d \geq 1$,
$$ \vz_d = \left(1, d \ \mathrm{ mod } \ 2, \ldots, d \ \mathrm{ mod } \ 2 \right) \in \R^{\N_{2d}}, $$
where $d \ \rm mod \ 2 = \left\{ \begin{array}{l}
1 \ \text{if} \ d \ \text{is odd} \\
0 \ \text{if} \ d \ \text{is even}
\end{array} \right.$ so that $\vz_d$ would be the truncated moment sequence corresponding to optimal solution $\mu^\star_d = (d \, \rm mod \, 2) \, \delta_1 + (1 - d \, \rm mod \, 2) \, \delta_0 = \delta_{d \, \rm mod \, 2}$.

Such a $(\vz_d)_d$ would obviously not converge when the degree $d$ tends to infinity.

{A natural question that arises then is: how does one modify the original GMP or the moment relaxations in order to retrieve pseudo-moment sequence convergence?}

A trivial way to fix this issue is to fix the value of $z_1$ in the constraints of the GMP:
\begin{subequations}
\begin{center}
\begin{minipage}{0.45\textwidth}
\begin{alignat}{4}
\op^\star_t := \ && \inf \ & \int (1-x)x \; d\mu(x) \label{opt:popfix} \\
&& \st \ & \ \ \ \mu \in \cM([0,1])_+ \nonumber \\
&&&\int 1 \; d\mu(x) = 1 \nonumber \\
&&&\int x \; d\mu(x) = t, \nonumber \\
\vphantom{\sum} \nonumber
\end{alignat}
\end{minipage}
\begin{minipage}{0.45\textwidth}
\begin{alignat}{4}
\op^d_t := \ && \ \inf \ & z_1 - z_2 \label{opt:popmfix} \\
&& \st \ & \ \vz = (z_k)_{0\leq k \leq 2d} \in \R^{\N_{2d}} \nonumber \\
&&& \mM_d(\vz) \succeq 0 \nonumber \\
&&& \mM_{d-1}\left(x(1-x)\vz\right) \succeq 0 \nonumber \\
&&& \mM_{d-1}\left((1 - x^2)\vz\right) \succeq 0 \nonumber \\
&&& z_0 = 1 \nonumber \\
&&& z_1 = t, \nonumber
\end{alignat}
\end{minipage}
\end{center}
\end{subequations}
for a given $t \in [0,1]$. Then, the only solution to \eqref{opt:popfix} is $$ \mu^\star = t \, \delta_1 + (1-t) \, \delta_0,$$
so that Assumption \ref{asm:unique} holds and one retrieves { pseudo-moment sequence} convergence (and finite convergence {still holds}), as the unique minimizer for \eqref{opt:popmfix} is
$$ \vz_d = (1,t,\ldots,t) = \left(\int x^k \; d\mu^\star(x)\right)_{0\leq k \leq 2d} \in \R^{\N_{2d}}. $$

\begin{rem}[Enforcing uniqueness]
\begin{itemize}
\item[]
\end{itemize}
It is not always possible to enforce Assumption \ref{asm:unique} as we did here without knowing the GMP solution in advance.
{Another common heuristics to obtain a unique solution {in the relaxations} consists in adding a penalty in the cost function: instead of minimizing $z_1 - z_2$ in \eqref{opt:popm}, one could minimize $z_1 - z_2 + \epsilon \, \mathrm{Tr} \, \mM_d(\vz)$, for some $\epsilon \in \R\setminus\{0\}$. Then, depending on the sign of $\epsilon$, one would obtain a unique solution $\vz_d = (1,\ldots,1)$ ($\epsilon < 0$) corresponding to $\mu^\star = \delta_1$, or $\vz_d = (1,0,\ldots,0)$ ($\epsilon>0)$ corresponding to $\mu^\star = \delta_0$.}
\end{rem}

We finally get to numerically implement the moment hierarchy corresponding to our illustrating example. In {the ``fixed''} cases, the behavior of the SDP solver is fixed, up to numerical errors. The interesting case is then problem \eqref{eq:pop}. We implemented the corresponding hierarchy of moment problems \eqref{opt:popm} with increasing relaxation degree $d \in \N^\star_{20}$. First, not displayed is the fact that we indeed invariably obtained $z_0 = 1.0000$ and $z_k = z_1$ (up to solver precision) for all $k \geq 2$. We implemented the hierarchy in two different frameworks:
\begin{itemize}
\item GloptiPoly~\cite{glopti}, a Matlab toolbox that models moment problems, interfaced with the SDP solvers SeDuMi (see the numerically computed values of $z_1$ in Table \ref{tab:sedumi}) and Mosek (see Table \ref{tab:mosek});
\item MomentOpt~\cite{momentopt}, an open source Julia module, interfaced with the SDP solver CSDP (see the numerically computed values of $z_1$ in Table \ref{tab:julia}).
\end{itemize}

\begin{table}[h!]

\begin{center}
\subfloat[with GloptiPoly and SeDuMi.\label{tab:sedumi}]{
\begin{tabular}{|c|c|c|c|c|c|c|c|c|c|c|}
\hline
$d$ & 1 & 2 & 3 & 4 & 5 & 6 & 7 & 8 & 9 & 10 \\
\hline
$z_1$ & 0.3974 & 0.3094 & 0.2228 & 0.2020 & 0.2103 & 0.1960 & 0.1440 & 0.1400 & 0.1369 & 0.0953 \\
\hline
\hline
$d$ & 11 & 12 & 13 & 14 & 15 & 16 & 17 & 18 & 19 & 20 \\
\hline
$z_1$ & 0.1088 & 0.1116 & 0.0982 & 0.0937 & 0.0851 & 0.0762 & 0.0716 & 0.0674 & 0.0649 & 0.0620 \\
\hline
\end{tabular}
}

\subfloat[with GloptiPoly and Mosek.\label{tab:mosek}]{
\begin{tabular}{|c|c|c|c|c|c|c|c|c|c|c|}
\hline
$d$ & 1 & 2 & 3 & 4 & 5 & 6 & 7 & 8 & 9 & 10 \\
\hline
$z_1$ & 0.3259 & 0.3034 & 0.1782 & 0.1312 & 0.1092 & 0.1236 & 0.0976 & 0.0898 & 0.0861 & 0.0830 \\
\hline
\hline
$d$ & 11 & 12 & 13 & 14 & 15 & 16 & 17 & 18 & 19 & 20 \\
\hline
$z_1$ & 0.0763 & 0.0744 & 0.0666 & 0.0611 & 0.0564 & 0.0528 & 0.0499 & 0.0473 & 0.0450 & 0.0428 \\
\hline
\end{tabular}
}

\subfloat[with MomentOpt and CSDP.\label{tab:julia}]{
\begin{tabular}{|c|c|c|c|c|c|c|c|c|c|c|}
\hline
$d$ & 1 & 2 & 3 & 4 & 5 & 6 & 7 & 8 & 9 & 10 \\
\hline
$z_1$ & 0.3741 & 0.0553 & 0.0772 & 0.0340 & 0.0418 & 0.0280 & 0.0169 & 0.0148 & 0.0121 & 0.0137 \\
\hline
\hline
$d$ & 11 & 12 & 13 & 14 & 15 & 16 & 17 & 18 & 19 & 20 \\
\hline
$z_1$ & 0.0079 & 0.0107 & 0.0082 & 0.0078 & 0.0071 & 0.0074 & 0.0075 & 0.0071 & 0.0059 & 0.0057 \\
\hline
\end{tabular}
}
\caption{Numerical solutions of problem \eqref{opt:popm}} \label{tab:results}
\end{center}
\end{table}

At first glance it would seem that the sequence $(z_{d,1})_d$ decreases when $d$ increases, but looking at specific values shows that it is not the case. What we can deduce from Table \ref{tab:results} is that the sequence $(z_{d,1})_d$ seems to converge to $0$, so that $(\vz_d)_d$ would converge towards $\mu_\infty = \delta_0$, but we have no proof that it is actually the case. Moreover, the sequence $\vz_d = (1,0,\ldots,0)$ is already feasible for the degree $d = 1$ relaxation, so that there is no reason, if the tendency were indeed to go to $z_1 = 0$, that such constraint is not satisfied already in the first steps of the hierarchy. Also, it is important to notice that the results we obtained are stable, meaning that running the same code always returns the same values for $z_1$. In other words, the obtained values \emph{only depend on the solver}. These tendencies then depend on how the interior point algorithm of the solver is implemented, which is out of the scope of the present work.
\color{black}

\section{Conclusion}

While usually it is only proved that the optimal values for the moment relaxations of the GMP monotonically converge towards the optimal value of the GMP, this note proves that the solutions of these moment relaxations converge to the actual solution of the GMP in $\ell^\infty$'s weak-$\ast$ topology, provided that some elementary assumptions hold. So far, such proof has only been provided for very specific instances of the GMP (see e.g. {\cite[Theorem 5.6(b)]{lasserre}}, \cite{volume}), while here the generic case is completely dealt with. Up to rescaling the problem (which also precludes ill behaviours in the numerical implementations), one can usually enforce quite easily that these conditions are met. As a byproduct of this proof, one also obtains strong duality both in the GMP and its corresponding hierarchy, generalizing the results of \cite{joszdual}. Among the open questions this note answers, one can cite the weak-$\ast$ convergence of the moment sequences associated to the OCP \cite{trelat} and set approximation problems in \cite{roa,inmpi,reachable}, as well as the strong duality between the GMP formulation in \cite{burgers} and its dual, regardless of the formulation of such dual. Finally, this note should facilitate most proofs of convergence and strong duality related to future applications of Lasserre's moment-SOS hierarchy.

\ifdefined\inmaster\else
\bibliographystyle{plain}
\bibliography{phd}
\end{document}
\fi